\documentclass[10pt]{amsart}

\usepackage{amsmath,amssymb,mathtools,listings,amsfonts,amsthm,stmaryrd,rotating,diagrams}
\usepackage{difftrees}

\newcommand*{\shuffle}
{{\,\begin{sideways}\begin{sideways}\begin{sideways}
$\tiny{\exists}$\end{sideways}\end{sideways}\end{sideways}\,}}

\newtheorem{theorem}{Theorem}
\newtheorem{rmk}{Remark}
\newtheorem{definition}{Definition}
\newtheorem{lemma}{Lemma}

\newtheorem{proposition}{Proposition}

\title[Arborification and renormalisation of RDEs]{Quasi-shuffle algebras and renormalisation of rough differential equations}

\author{Yvain Bruned}
\address{Department of Mathematics, 
		Imperial College London, 
		London SW7 2AZ, UK.}
\email{y.bruned@imperial.ac.uk}

\author{Charles Curry}
\address{Department of Mathematical Sciences, 
		Norwegian University of Science and Technology (NTNU), 
		7491 Trondheim, Norway.}
\email{charles.curry@ntnu.no}

\author{Kurusch Ebrahimi-Fard}
\address{Department of Mathematical Sciences, 
		Norwegian University of Science and Technology (NTNU), 
		7491 Trondheim, Norway.}
\email{kurusch.ebrahimi-fard@ntnu.no}         
\urladdr{https://folk.ntnu.no/kurusche/}

\begin{document}


\begin{abstract}
The objective of this work is to compare several approaches to the process of renormalisation in the context of rough differential equations using the substitution bialgebra on rooted trees known from backward error analysis of $B$-series. For this purpose, we present a so-called arborification of the Hoffman--Ihara theory of quasi-shuffle algebra automorphisms. The latter are induced by formal power series, which can be seen to be special cases of the cointeraction of two Hopf algebra structures on rooted forests. In particular, the arborification of Hoffman's exponential map, which defines a Hopf algebra isomorphism between the shuffle and quasi-shuffle Hopf algebra, leads to a canonical renormalisation that coincides with Marcus' canonical extension for semimartingale driving signals. This is contrasted with the canonical geometric rough path of Hairer and Kelly by means of a recursive formula defined in terms of the coaction of the substitution bialgebra.
\end{abstract}


\maketitle

\tableofcontents





\section{Introduction}
\label{sect:intro}
 
We consider controlled differential equations of the form
\begin{equation}
\label{eq:CODE}
	dY_t = \sum_{i=0}^d f_i(Y_t)dX^i_t,
\end{equation}
where the $f_i$ are vector fields on $\mathbb{R}^{e}$ and $X:[0,T] \rightarrow \mathbb{R}^{d+1}$ is the driving signal. For sufficiently regular $f_i$ and $X$, e.g., Lipschitz continuous vector fields and $X$ of bounded variation, a unique solution exists, for which we have the formal expansion
\begin{equation}
\label{eq:CODEsolution}
	Y_t =  Y_s + \sum_{\tau \in \mathcal{T}_A} \frac{1}{\sigma(\tau)} \mathfrak{F}_f[\tau](Y_s) X^{\tau}_{st}, 
\end{equation}
where the sum is indexed by the set of non-planar rooted trees $\tau \in \mathcal{T}_A$ with vertex decorations in the set $A=\{0,1,\ldots,d\}$. The symmetry factor $\sigma(\tau)$ is a combinatorial coefficient defined in \textsection 2.2. The $\mathfrak{F}_f[\tau]$ are the so-called elementary differentials corresponding to the vector fields $f_i$ and the rooted tree $\tau \in \mathcal{T}_A$ with decorations in $A$. The $X^{\tau}_{st}$ are multiple integrals of the driving signal over domains indexed by $\tau \in \mathcal{T}_A$. For instance, as the domain of integration specified by a (decorated) non-planar rooted tree involves a partial order, the integral $X^{\tau}_{st}$ corresponding to the tree  ${}_{\phantom{i_1}}^{i_2}\hspace{-.08cm} \scalebox{0.5}{\aababb}^{\!i_3}_{\!\!i_1} \in \mathcal{T}_A$ is given by
$$
	X^{\scalebox{0.3}{\aababb}}_{st}=
	\int_{s\le t_2 \le t_1\le t \atop s\le t_3 \le t_1\le t }
	dX^{i_1}_{t_1}dX^{i_2}_{t_2} dX^{i_3}_{t_3}=
	\int_{s}^{t}
	dX^{i_1}_{t_1}\int_{s}^{t_1}dX^{i_2}_{t_2} \int_{s}^{t_1}dX^{i_3}_{t_3}.
$$ 
The reader is referred to Gubinelli's work \cite{Gubinelli10,Gubinelli11} for more details. Under the assumption that iterated integrals obey classical integration by parts, the domain may be split up according to all the linear extensions of the partial order specified by the rooted tree, such that $X^{\tau}_{st}$ can be resolved into linear combinations of strictly iterated integrals. For the above example this yields
$$
	X^{\scalebox{0.3}{\aababb}}_{st}=
	\int_{s}^{t}
	dX^{i_1}_{t_1}\int_{s}^{t_1}dX^{i_3}_{t_2} \int_{s}^{t_2}dX^{i_2}_{t_3}+
	\int_{s}^{t}
	dX^{i_1}_{t_1}\int_{s}^{t_1}dX^{i_2}_{t_2} \int_{s}^{t_2}dX^{i_3}_{t_3}.
$$
Collecting the coefficients of the resulting iterated integrals in the expansion \eqref{eq:CODEsolution} yields Chen's classical word series
\[
	Y_t = Y_s + \sum_{w \in T(\mathbb{R}^{d+1})} V_f (w)(Y_s) X^w_{st}
\]
as solution to the controlled differential equation \eqref{eq:CODE}, where the sum is now over words in the tensor algebra $T(\mathbb{R}^{d+1})$ over $\mathbb{R}^{d+1}$ with basis vectors $\{e_0,\ldots,e_d\}$. The key to Lyons' theory of rough differential equations \cite{Lyons98,LCL07} is that for rough signals, say $\alpha$-H\"{o}lder continuous for some $0<\alpha<\frac{1}{2}$, a complete set $\mathbb{X}_{st}$ of iterated integrals $X^w_{st}$ may be constructed canonically from a finite subset of iterated integrals. The amount of multiple integrals required to complete the construction depends on the regularity of the driving path. The differential equation is then interpreted as being driven by the rough path $\mathbb{X}_{st}$:
\[
	d\mathbb{Y} = f(\mathbb{Y})d\mathbb{X},
\]
where existence and uniqueness is determined by means of fixed point arguments in a space of suitably controlled rough paths. 

The assumption that the iterated integrals obey a conventional integration by parts relation is at times restrictive, and led Gubinelli \cite{Gubinelli10,Gubinelli11} to introduce the notion of a branched rough path as a collection of iterated integrals $X^{\tau}_{st}$ indexed by rooted trees. Differential equations driven by branched rough paths are interpreted in a similar manner. We remark here that from an algebraic point of view, a geometric (branched) rough path is nothing other than a two-parameter family of characters over the shuffle (Butcher--Connes--Kreimer) Hopf algebra (of rooted trees), obeying certain analytical conditions. The reader is referred to Hairer's and Kelly's article \cite{HaiKel15} for detailed accounts on Lyons' and Gubinelli's rough paths.

The importance of the process of renormalisation in Hairer's theory of regularity structures and its relation to rough path theory \cite{BruHaiZam16,FrizHairer14,Hairer14} has led to recent interest in the analogous role of renormalisation of rough differential equations \cite{BCFP17}. For example, suppose that $B_t$ is a Brownian path that we wish to lift to a (branched) rough path $\mathbb{B}_{st}$. This lift may be accomplished by stochastic integration, but in doing so we introduce a lifting parameter corresponding to the point where the integrand is evaluated in each subinterval of the Riemann sums. Taking the left endpoint results in an It\^{o} integral, whilst taking the midpoint gives a Stratonovich integral. Different rough path lifts are related by a renormalisation procedure \cite{BCFP17}, e.g., the usual It\^o--Stratonovich conversion is given by 
\[
	\mathbb{B}^{\mathrm{Strat}}_{st} = \mathbb{B}^{\mathrm{Ito}}_{st} + \frac{1}{2}I(t-s).
\] 
The importance of deciding upon and relating such lifts in practical stochastic modeling has resulted in several alternative descriptions of this apparent ambiguity in defining the stochastic integral, namely
\begin{enumerate}

	\item The canonical extensions of McShane and Marcus \cite{Marcus78,Marcus81,McShane75}

	\item Hoffman's exponential and logarithm \cite{CEMW14,EMPW15a,EMPW15b}

	\item Translations of rough paths by Bruned et al.~\cite{BCFP17}


\end{enumerate}

\medskip

In this paper we show that the substitution bialgebra on non-planar rooted trees defined in \cite{CEFM11}, which here is extended to decorated trees, provides a common method to understand these approaches. Hoffman's exponential and logarithm follow from a general theory showing that every formal power series induces a quasi-shuffle algebra automorphism, and that the composition of two such automorphisms is equivalent to the automorphism induced by the composition of the power series \cite{Hoffman00,HoffmanIhara16}. Among our main results is the description of an ``arborified" Hoffman exponential using the substitution bialgebra. In a nutshell, it corresponds to a simple replacement of the usual factorials in the exponential map by rooted tree factorials. Then we show that Marcus canonical extension \cite{Marcus78,Marcus81} is an instance of the adjoint of the arborified Hoffman exponential. 

Bruned et al.~\cite{BCFP17} show that the effect of translations on a rough differential equation can be measured by a change in the driving vector field $f$, and give a procedure to express multiple integrals with respect to the translated path in terms of multiple integrals of the original path. The mechanism relies on an extension to the decorated case of the cointeraction of the substitution bialgebra on the Butcher--Connes--Kreimer Hopf algebra, originally shown in \cite{CEFM11}. Among the possible lifts of Brownian paths, the Stratonovich lift is canonical in the sense that it gives rise to iterated integrals obeying the usual rules of calculus -- the resulting rough path is geometric in Lyons' terminology. Whenever a branched rough path is constructed by integration of paths obeying a quasi-shuffle law, a canonical renormalisation is obtained by the translation associated to the inverse tree factorial character. We show that in the case of semimartingale integrators, this coincides with the canonical extension of Marcus.

We conclude with a brief discussion of an alternative approach, where a canonical geometric rough path is constructed above a given branched rough path. This approach was initiated by Hairer and Kelly in \cite{HaiKel15}, and a related strategy has recently been pursued in \cite{BoeChe17}. The idea is to expand the vector space on which the signal is considered to live by defining a Hopf algebra morphism from the Butcher--Connes--Kreimer Hopf algebra to the tensor algebra with shuffle product defined over the set of rooted trees. Superficially this is unrelated to the substitution algebra approach, and we confine ourselves to the observation that the Hairer--Kelly map can be understood using the right comodule structure introduced in \cite{CEFM11}. 

\medskip

The paper is organised into six sections. The next section gives a brief overview of the necessary algebraic structures. In \textsection 3 we review how the effects of certain modifications of vector fields (equivalently, translations of rough paths) can be understood algebraically. Hoffman's exponential is given in \textsection 4, and extended to the arborified case. A formula for its adjoint is also given. We follow this with a discussion of the Marcus canonical extension in \textsection 5, showing that this can be understood in terms of the adjoint of the arborified Hoffman exponential. We conclude in \textsection 6 with a brief discussion of the different approach by Hairer and Kelly.

\vspace{0.5cm}

{\bf{Acknowledgements}}: The research on this paper was partially supported by the Norwegian Research Council (project 231632). The first named author thanks Martin Hairer for financial support via
a Leverhulm Trust leadership award. The second named author would like to thank Dominique Manchon for helpful discussions.

\vspace{0.5cm}


\section{Rooted trees, words and Hopf algebras}
\label{sect:TreesWordsHopf}

In what follows $k$ denotes the ground field of characteristic zero over which all algebraic structures are considered.


\subsection{Pre-Lie algebra}
\label{ssect:preLie}

 A left pre-Lie algebra $(P, \triangleright)$ is a $k$-vector space $P$ together with a bilinear product $\triangleright \colon P \otimes P \to P$ which satisfies the left pre-Lie identity  \cite{Cartier10,Manchon11}
\begin{equation}
\label{eq:lpre-Lie}
	(x \triangleright y)\triangleright z - x\triangleright (y\triangleright z)
	= (y \triangleright x) \triangleright z - y \triangleright (x \triangleright z),
\end{equation}
for any elements $x,y,z \in P$. An analogous definition exists for right pre-Lie algebra. Note that \eqref{eq:lpre-Lie} writes as $L_{[x,y]\triangleright}(z)=[L_{x\triangleright},L_{y\triangleright}](z)$, where $L_{a\triangleright}(b):=a \rhd b$ and $L^n_{a\triangleright}(b):=a\triangleright (L^{n-1}_{a\triangleright }(b))$, $L^0_{a \triangleright }:=\mathrm{id}$, for any $a,b \in P$. The bracket $[x,y]:=x \triangleright y - y \triangleright x$ satisfies the Jacobi identity. Let $(P_1, \triangleright_1)$ and $(P_2, \triangleright_2)$ be two pre-Lie algebras. A pre-Lie morphism is a $k$-linear map $\psi: P_1 \to P_2$, such that $\psi(x \triangleright_1 y)=\psi(x) \triangleright_2 \psi(y)$. A natural example of pre-Lie algebra is given in terms of a differentiable manifold $M$ with a flat and torsion-free connection. The corresponding covariant derivation $\nabla$ on the space $\chi(M)$ of vector fields on $M$ provides it with a left pre-lie algebra structure, which is defined by $f \triangleright g=\nabla_fg$, by virtue of the two equalities $\nabla_fg-\nabla_gf=[f,g]$ and $\nabla_{[f,g]}=[\nabla_f,\nabla_g]$, which express vanishing of torsion and curvature, respectively. Let $M=\mathbb{R}^n$ with its standard flat connection. For $f(x)=\sum_{i=1}^nf^i(x)\frac{\partial}{\partial x^i}$ and $g(x)=\sum_{i=1}^ng^i(x)\frac{\partial}{\partial x^i}$ it follows that
\begin{equation}
\label{preLieVect1}
	(f \triangleright g) (x)=\sum_{i=1}^n\Big(\sum_{j=1}^n 
	f^j(x)\frac{\partial}{\partial x^j}g^i(x)\Big)\frac{\partial}{\partial x^i}.
\end{equation}
Consider the initial value problem $\dot{y}_t=f(y_t)$, $y_{t=0}$, where $f$ is a vector field on $\mathbb{R}^n$. The solution can be written in terms of the flow, by using the pre-Lie product \eqref{preLieVect1} of vector fields \cite{Chapoton01}
\begin{equation}
\label{preLieVect2}
	y_t = (e^{(L_{tf \triangleright})}\circ \mathrm{id})y_0
	      =y_0 + \sum_{n>0} \frac{t^n}{n!}L^{n-1}_{f \triangleright}(f)(y_0).
\end{equation}
This can be seen from iterating the equivalent integral equation 
\begin{equation}
\label{solution}
	y_t = y_0 + \int_0^t f(y_s)ds.
\end{equation}
Indeed, observe that for $h: \mathbb{R}^n \to \mathbb{R}^n$, we have that $\dot{h}(y_t) = (f(y_t) \cdot \nabla) h(y_t)$. From this we see that $h(y_t)=h(y_0) + \int_0^t (f(y_s) \cdot \nabla) h(y_s) ds$. For $h=\mathrm{id}$ it follows that $(f(y_s) \cdot \nabla) y_s =f(y_s)$ and $(f(y_s) \cdot \nabla)(f(y_s) \cdot \nabla) y_s =(f(y_s) \cdot \nabla)f(y_s)=(f\triangleright f)(y_s)$. Applying this to \eqref{solution} and iterating yields
\begin{equation}
\label{solutionFlow}
	y_t= y_0 + f(y_0)t + (f \triangleright f)(y_0) \frac{t^2}{2!} 
	+ (f \triangleright (f \triangleright f))(y_0) \frac{t^3}{3!} + \cdots. 
\end{equation}


\subsubsection{Free pre-Lie algebra}
\label{sssect:freepre-Lie}

Recall that a rooted tree $\tau$ consists of vertices and nonintersecting oriented edges. All but one of the vertices have exactly one outgoing edge and an arbitrary number of incoming ones. The root is the one vertex with no outgoing edge and it is drawn on the bottom of the tree. The leaves are the only vertices without incoming edges. The set of non-planar rooted trees is denoted by 
\[
	\mathcal{T} = \left\{\begin{array}{c}
		\ab, \aabb,\aaabbb, \aababb, \aaaabbbb,\aaababbb,\aaabbabb,\aabababb,\ldots
 			\end{array}
			\right\}.
\]
The sets of vertices and edges of $\tau \in \mathcal{T}$ are denoted by $V(\tau)$ respectively $E(\tau)$. 

Chapoton and Livernet \cite{ChaLiv01} showed that the basis of the free pre-Lie algebra $\mathcal{P}(\ab)$ in one generator can be expressed in terms of undecorated, non-planar rooted trees. The pre-Lie product in  $\mathcal{P}(\ab)$ is given in terms of grafting, that is, $\tau_1 \triangleright \tau_2$ is given by summing over all trees resulting from grafting the tree $\tau_1$ successively to each vertex of $\tau_2$: 
\begin{equation}
\label{treepreLie}
	\tau_1 \triangleright \tau_2 
	= \sum_{v \in V(\tau_2)} \tau_1 \to_{v} \tau_2 
	=\sum_{\tau \in \mathcal{T}} M^{\tau_1}_{\tau_2;\tau}\tau,
\end{equation}
where $\to_{v}$ denotes the grafting of the root of $\tau_1$ via a new edge to vertex $v$ of $\tau_2$. The coefficient $M^{\tau_1}_{\tau_2;\tau}$ is the number of vertices of $\tau_2$ such that grafting $\tau_1$ on $\tau_2$ at such a vertex gives the tree $\tau$, e.g.,
\begin{equation*}
	\ab  \triangleright \ab =  \aabb \ , 
	\quad
	 \ab  \triangleright \aabb = \aababb + \aaabbb \ ,  
	  \quad
	 \aabb \triangleright \aabb = \aaaabbbb + \aaabbabb \ , 
	  \quad
	 \ab \triangleright\! \aababb =  2\aaabbabb + \aabababb.
\end{equation*}
Indeed, given trees $\tau_1,\tau_2,\tau_3 \in \mathcal{T}$ the expression $\tau_1 \triangleright (\tau_2 \triangleright \tau_3) - (\tau_1 \triangleright \tau_2)\triangleright \tau_3$ is the sum of all the trees obtained by grafting $\tau_1$ and $\tau_2$ at two distinct vertices of $\tau_3$. It is thus symmetric in $\tau_1$ and $\tau_2$, and hence the linear span of rooted trees is a left pre-Lie algebra. Let $\Omega$ be a set. The free pre-Lie algebra $\mathcal{P}(\Omega)$ in $|\Omega|$ generators amounts to non-planar rooted trees decorated by elements from $\Omega$. We denote the set of $\Omega$-decorated non-planar rooted trees by $\mathcal{T}_\Omega$. The free pre-Lie algebra $\mathcal{P}(\Omega)$ satisfies the universal property: for any pre-Lie algebra $(Q,\curvearrowright)$ and any mapping $v: \Omega \to Q$ there exists unique pre-Lie algebra morphism $\tilde{v}:  \mathcal{P}(\Omega) \to Q$, such that $f=\tilde{f} \circ \iota$, where $\iota: \Omega \to \mathcal{P}(\Omega)$. The elementary differential $\mathfrak{F}_f$ in \eqref{eq:CODEsolution} is such a pre-Lie algebra morphism, extending the map $v(\ab_i):=f_i$ for $i \in \Omega=\{0,\ldots,d\}$.


\subsection{BCK Hopf algebra and GL product}
\label{ssect:BCKGL}

A $\Omega$-decorated rooted forest is a finite collection $F=(\tau_1,\ldots, \tau_n)$ of rooted trees from $\mathcal{T}_\Omega$, which we denote by the commutative product $\tau_1 \cdots \tau_n$. The set of $\Omega$-decorated forests is denoted $\mathcal{F}_\Omega$ and contains the empty forest $\mathbf{1}=\{\emptyset\} \in \mathcal{F}_\Omega$. The operator $B^i_+$, $i \in \Omega$, associates to the forest $F$ the tree $B^i_+(F) \in \mathcal{T}_\Omega$ obtained by grafting each tree in $F$ on a common new root decorated by $i$. The unique rooted tree $\ab_i=B^i_+(\mathbf{1})$. Recall the definition of several important numbers associated to a rooted forest $F= \tau_1 \cdots \tau_n$ respectively the tree $\tau=B^i_+(F)$. The number of vertices $|V(\tau)|$ is the sum given by $|\tau|:=1+\sum_{j=1}^{n} |\tau_j|$. The tree factorial is recursively defined with respect to the number of vertices, i.e.,  by $\ab_i !=1$ and $\tau!=|\tau|\prod_{j=1}^n\tau_j!$. It is multiplicatively extended to forests, $F! = \tau_1! \cdots \tau_n!$. The internal symmetry factor $\sigma(F)=\prod_{j=1}^n|\mathrm{Aut}(\tau_j)|$. The Connes--Moscovici coefficient $\mathrm{cm}(\tau)$ of a tree $\tau$ is defined by
\begin{equation}
\label{CM}
	\mathrm{cm}(\tau)=\frac{|\tau|!}{\tau!\sigma(\tau)}.
\end{equation}
$$
	\mathrm{cm}(\ab)=1,
	\quad
	\mathrm{cm}(\hspace{-0.1cm}\begin{array}{c}\aabb\end{array}\hspace{-0.1cm})=1,
	\quad
	\mathrm{cm}(\hspace{-0.1cm}\begin{array}{c}\aababb\end{array}\hspace{-0.1cm})=1,
	\quad
	\mathrm{cm}(\hspace{-0.1cm}\begin{array}{c} \aaababbb\end{array}\hspace{-0.1cm})=1,
	\quad
	\mathrm{cm}(\hspace{-0.1cm}\begin{array}{c}\aaabbabb\end{array}\hspace{-0.1cm})=3.
$$
For the last two coefficients we used that the symmetry factor and tree factorial of the tree $B_+(B_+(B_+(\mathbf{1})B_+(\mathbf{1})))$ are $2$ respectively $12$, whereas for the tree $B_+(B_+(B_+(\mathbf{1})) B_+(\mathbf{1}))$ they are one respectively $8$. An important result is the following lemma:

\begin{lemma}\cite{Brouder00}\label{cm}
The Connes--Moscovici coefficient of $\tau \in \mathcal{T}$ is characterised by
\[
	L_{\ab\,\triangleright}^{n-1}(\ab) = \sum_{\tau \in \mathcal{T} \atop |\tau|=n} \mathrm{cm}(\tau)\tau,
\]
where the sum on the right runs over all trees in $\mathcal{T}$ with exactly $n$ vertices.
\end{lemma}

\begin{align*}
	L_{\ab\,\triangleright}^{1}(\ab)&=\ab \triangleright \ab = \aabb\\
	L_{\ab\,\triangleright}^{2}(\ab)&=\ab \triangleright(\ab \triangleright \ab) = \aababb + \aaabbb\\
	L_{\ab\,\triangleright}^{3}(\ab)&=\ab \triangleright(\ab \triangleright(\ab \triangleright \ab))
	= 3\aaabbabb + \aabababb + \aaaabbbb + \aaababbb\\
\end{align*}

\medskip

The $\Omega$-decorated Butcher--Connes--Kreimer Hopf algebra $\mathcal{H}^\Omega_{BCK}$ of rooted forests is the free unital commutative $k$-algebra on the linear space $\mathcal{T}_\Omega$ \cite{CK98,Manchon08}. It is graded by the number of vertices. The non-cocommutative coproduct is defined for the tree $\tau=B^i_+(F) \in \mathcal{T}_\Omega$ by
$$
	\Delta(\tau) := \tau \otimes \mathbf{1} + (\mathrm{id} \otimes B^i_+)\Delta(F), 
$$ 
where $\Delta(F) := \Delta( \tau_1) \cdots \Delta( \tau_n)$ and $\Delta(\mathbf{1}):=\mathbf{1} \otimes \mathbf{1}$. Observe that the space spanned by ladder trees $B^{i_1}_+ \circ \cdots \circ B^{i_l}_+(\mathbf{1})$ forms a cocommutative Hopf subalgebra in $\mathcal{H}^\Omega_{BCK}$ with the simple coproduct
$$
	\Delta(B^{i_1}_+ \circ \cdots \circ B^{i_l}_+(\mathbf{1}))=
	\sum_{m = 0}^l B^{i_1}_+ \circ \cdots \circ B^{i_m}_+(\mathbf{1}) 
	\otimes B^{i_{m+1}}_+ \circ \cdots \circ B^{i_{l}}_+(\mathbf{1}),
$$
where it is understood that $B^{i_1}_+ \circ \cdots \circ B^{i_0}_+(\mathbf{1}) = B^{i_{l+1}}_+ \circ \cdots \circ B^{i_{l}}_+(\mathbf{1}) =\mathbf{1}$.

We denote by $(d_F)_{F \in \mathcal{F}_\Omega}$ the normalised dual basis in the graded dual ${\mathcal  H}_{BCK}^{\Omega *}$ of the forest basis of ${\mathcal  H}^\Omega_{BCK}$, i.e., $\langle d_{F},G \rangle = \sigma(F)$ when the forests $F$ and $G$ coincide, and zero else. For any tree $\tau \in \mathcal{T}_\Omega$ the corresponding $d_\tau$ is an infinitesimal character over ${\mathcal  H}^\Omega_{BCK}$, in other words, it is a primitive element of ${\mathcal  H}_{BCK}^{\Omega *}$. The convolution product of ${\mathcal  H}_{BCK}^{\Omega *}$ gives rise to the Lie bracket $[d_{\tau_1},d_{\tau_2}]=d_{\tau_1} \ast d_{\tau_2} - d_{\tau_2} \ast d_{\tau_1}$ for primitives in ${\mathcal  H}_{BCK}^{\Omega *}$. The link with the pre-Lie product \eqref{treepreLie} on rooted trees follows from \cite{CK98}
\begin{equation}
	d_{\tau_1} \ast d_{\tau_2} - d_{\tau_2} \ast d_{\tau_1} 
	= d_{{\tau_1} \triangleright {\tau_2} - {\tau_2} \triangleright {\tau_1}}.
\end{equation}
By the Cartier--Milnor--Moore theorem ${\mathcal  H}_{BCK}^{\Omega *}$ is isomorphic as a Hopf algebra to the enveloping algebra $\mathcal{U}(\mathfrak{g}_\Omega)$ \cite{Cartier07}, where $\mathfrak{g}_\Omega:=\mathrm{Prim}({\mathcal  H}_{BCK}^{\Omega *})$ is the Lie algebra spanned by the $d_{\tau}$'s for rooted trees ${\tau} \in \mathcal{T}_\Omega$. Guin and Oudom \cite{OudomGuin08} extended the pre-Lie product on $\mathcal{P}(\Omega)$ to the symmetric module $\mathcal{S}(\mathfrak{g}_\Omega)$. Then they defined another associative product on $\mathcal{S}(\mathfrak{g}_\Omega)$ and showed that $(\mathcal{S}(\mathfrak{g}_\Omega),\Delta,*)$ is isomorphic as a Hopf algebra to $\mathcal{U}(\mathfrak{g}_\Omega)$. The associative product is defined for $F,G \in \mathcal{S}(\mathfrak{g}_\Omega)$ by 
\begin{equation}
\label{def:preLie}
	F * G := F_{(1)}(F_{(2)} \triangleright G),
\end{equation}
where, using Sweedler's notation, $\Delta_\shuffle(F)=\sum F_{(1)} \otimes F_{(2)}$ is the usual unshuffle coproduct. This product \eqref{def:preLie} can be written in terms of the $B_+$ operation, i.e.,  $F \ast G=B^i_-(F \triangleright B^i_+(G))$, i.e., by summing all the trees obtained by grafting the trees in the forest $F$ on the tree $B^i_+(G)$ at various places, and then removing the root of each component of the sum, which is the definition of the $B^i_-$-operation. Note that the decoration of $B_\pm$ does not matter here. For trees $\tau_1, \tau_2 \in \mathcal{T}_\Omega$ we find
$$
	\tau_1 * \tau_2 = \tau_1\tau_2 + \tau_1 \triangleright \tau_2, 
$$ 
such that, for instance
$$
	\ab_i * \ab_j = \ab_i\ \ab_j + \ab_i \triangleright \ab_j =  \ab_i\ \ab_j + \aabb^i_j. 
$$
As a result the product $\ast$ on $\mathcal  S(\mathfrak{g}_\Omega)$ coincides with the Grossman--Larson product
\begin{equation}
\label{GL}
	d_{F \ast G}=d_F \ast d_G,
\end{equation}
where the product on the right-hand side is the convolution product in ${\mathcal  H}_{BCK}^{\Omega *}$.


\subsection{Shuffle and quasi-shuffle Hopf algebras}
\label{ssect:shuffle}

Algebraically, products of iterated Stratonovich integrals can be described using the usual shuffle product \cite{Reutenauer93}, whilst for the other stochastic integrals we obtain a quasi-shuffle relation. We refer the reader to Gaines' 1994 paper \cite{Gaines94} for more details. 

\smallskip

Following Hoffman \cite{Hoffman00} a quasi-shuffle algebra is defined on a locally finite alphabet $A$. By $A^*$ we denote the monoid of words $w=i_1 \cdots i_l$ generated by the letters from $A$ with concatenation as associative product. Moreover, we assume that $A$ itself is a commutative semigroup with binary product $[-\ -] \colon  A \times A \to A$. Commutativity and associativity of the latter allows us to write $[{i_1} \cdots  {i_{n}}] :=[{i_1} [\cdots [{i_{n-1}}{i_{n}}]]\cdots]$. The free non-commutative $k$-algebra of words $w={i_1} \cdots {i_l}$ over the alphabet $A$ is denoted $k\langle A\rangle$. The empty word is $\mathbf{1} \in k\langle A\rangle$. The length $|u|$ of a word $u \in k\langle A\rangle$ is defined by its number of letters. The commutative and associative quasi-shuffle product on words $w,v \in k\langle A\rangle$ is defined by 
\begin{enumerate}
 	\item $\mathbf{1} \star v :=v \star \mathbf{1} := v$,
 	\item $i v \star j w := i(v \star j w) + j(i v \star w) + [i  j] (v \star w)$,
\end{enumerate}
where $i,j \in A$. We denote the quasi-shuffle algebra by $\mathcal{H}_\star:=(k\langle A\rangle,\star)$. Hoffman \cite{Hoffman00} showed that $\mathcal{H}_\star=(k\langle A\rangle,\triangle,\star)$ is a Hopf algebra with respect to the deconcatenation coproduct $\triangle$. If $[i j]=0$ for any letters $i,j \in A$, then the quasi-shuffle product reduces to the ordinary shuffle product $\shuffle \colon k\langle A\rangle \times k\langle A\rangle \to k\langle A\rangle$
$$
	i v \shuffle j w:= i(v \shuffle jw) + j(i v \shuffle w),
$$
and $H_\star$ turns into the classical shuffle Hopf algebra $\mathcal{H}_\shuffle:=(k\langle A\rangle,\triangle,\shuffle)$. Further below we shall see another, more surprising link between $H_\star$ and $H_\shuffle$ due to Hoffman \cite{Hoffman00}. For $i \in A$, define the operator $R^i : \mathcal{H}_\star \to \mathcal{H}_\star$ by $R^i (w)=w i$. It verifies with respect to the deconcatenation coproduct:
$$
	\triangle\big(R^i(w)\big) = R^i(w)\otimes\mathbf{1} + (\mathrm{id} \otimes R^i)\triangle(w).
$$


\subsubsection{Arborifications}
\label{sssect:arborif}

As before, let $A$ be the alphabet which is also a commutative semigroup with binary product $[-\ -]$ introduced above. Generally, the process of (contracting) arborification is given by a surjective Hopf algebra morphism from the $A$-decorated Butcher--Connes--Kreimer Hopf algebra $\mathcal{H}^A_{BCK}$ onto the {(quasi-)}\- shuffle Hopf algebra ($\mathcal{H}_{\star}$) $\mathcal{H}_{\shuffle}$ defined over the alphabet $A$. See \cite{EMF16,FM17} for details. In the shuffle case, the arborification morphism $\mathfrak{a}: \mathcal{H}^A_{BCK} \to \mathcal{H}_{\shuffle}$ is defined by $\mathfrak{a} \circ B^i_+ := R^i \circ \mathfrak{a}$, i.e., 
\begin{equation}
\label{arborif}
	\mathfrak{a}(B^i_+(F))=(\mathfrak{a}(\tau_1 ) \shuffle \cdots \shuffle \mathfrak{a}(\tau_n)) i.
\end{equation}

In the quasi-shuffle case, the contracting arborification  morphism $\mathfrak{a}^c : \mathcal{H}^A_{BCK} \to \mathcal{H}_{\star}$ is defined analogously by $\mathfrak{a}^c \circ B^i_+ := R^i \circ \mathfrak{a}^c$, that is,
\begin{equation}
\label{arborifc}
	\mathfrak{a}^c(B^i_+(F))=(\mathfrak{a}^c(\tau_1 ) \star \cdots \star \mathfrak{a}^c(\tau_n))i.
\end{equation}
For instance, both arborifications map decorated ladder trees to single words, e.g., $\mathfrak{a}(B^{i_1}_+ \circ \cdots \circ B^{i_l}_+(\mathbf{1}))=\mathfrak{a}^c(B^{i_1}_+ \circ \cdots \circ B^{i_l}_+(\mathbf{1}))={i_l} \cdots {i_1}$. Non-ladder trees are mapped to linear combinations of words, e.g.,   
$$
	\mathfrak{a}^c(B^l_+(\ab_i \ \ab_j))
	=(\mathfrak{a}(\ab_i )  \star \mathfrak{a}(\ab_j))l 
	=(i \star j)l=ijl + jil + [i j]l
$$
and 
$$
	\mathfrak{a}(B^l_+(\ab_i \ \ab_j))
	=(\mathfrak{a}(\ab_i )  \shuffle \mathfrak{a}(\ab_j))l 
	=(i \shuffle j)l=ijl + jil.
$$
Observe that the difference between the two maps are extra terms coming from contractions.


\subsection{Substitution bialgebra}
\label{ssect:subsHopf}

Let $A$ be the locally finite commutative semigroup from above. Let $\mathcal{F}^A_+$ be the set of $A$-decorated forests excluding the empty forest, and consider the commutative polynomial algebra $\mathcal{H}^A_+$ graded by the number of edges. A subforest of a rooted tree $\tau$ is a collection $(\tau_1,\ldots ,\tau_n)$ of pairwise disjoint rooted subtrees of $\tau$. The decorations of each tree $\tau_j$ is induced by the decoration of $\tau$. In particular, two subtrees of a subforest cannot have any common vertex. A full subforest $\overline{F}=\tau_1\cdots \tau_k$ of $\tau$ is such that $V(\overline{F}) = V(\tau)$. The coproduct on $\mathcal{H}^A_+$ is defined in terms of extractions and contractions
\begin{equation}
\label{eq:subCoprod}
	\delta^+(\tau)=\sum_{\overline{F} \subseteq \tau} \overline{F} \otimes \tau/\overline{F}.
\end{equation}
Here the sum is over all $A$-decorated full subforests $\overline{F}=\tau_1\cdots \tau_k$ and $\tau/\overline{F}$ denotes the $A$-decorated rooted tree in which all connected components $\tau_j$ of $\overline{F}$ have been contracted and replaced by vertices $\ab_{[\tau_j]}$. Hence, for $\tau \in \mathcal{T}_A$ the coproduct $\delta^+(\tau)$ is linear on the right. The decoration $[\tau_j]:=[{j_1} \cdots {j_{|\tau_j|}}] \in A$, where $j_1,\ldots, j_{|\tau_j|}$ are the decorations of the vertices of the subtree $\tau_j$. Note that if $\tau_j = \ab_{i_j}$ then $\ab_{[\tau_j]}=\ab_{i_j}$, since $[{i_j}]={i_j}$. A few examples show that the number of edges is preserved
\begin{eqnarray*}
	\delta^+(\ab_i)	\!\!\!&=&\!\!	\ab_i \otimes \ab_i\\	
	\delta^+(\!\!\begin{array}{c}\aabb^i_j\end{array}\!\!\!)	
	\!\!\!&=&\hspace{-.3cm} \begin{array}{c}\aabb^i_j \end{array}\!\!\! \otimes \ab_{[i j]} 
			+ \ab_i\ \ab_j \otimes \!\! \begin{array}{c}\aabb^i_j\end{array}\\
	\delta^+(\hspace{-.3cm}\begin{array}{c}{\phantom{\aabb}}^{i_1}\aababb^{i_2}_{\!\!\!i_3}\end{array}\hspace{-.2cm})	
	\!\!\!&=&\hspace{-.6cm} \begin{array}{c}{\phantom{\aabb}}^{i_1}\aababb^{i_2}_{\!\!\!i_3}\end{array} \hspace{-.2cm}\otimes \ab_{[{i_1}{i_2}{i_3}]} 
	+ \ab_{{i_1}}\ab_{{i_2}}\ab_{{i_3}} \otimes\hspace{-.4cm} \begin{array}{c}{\phantom{\aabb}}^{i_1}\aababb^{i_2}_{\!\!\!i_3}\end{array} \hspace{-.2cm}
	+ \hspace{-.1cm} \begin{array}{c}\aabb^{i_1}_{i_3}\end{array} \hspace{-.2cm}\ab_{{i_2}} \otimes \hspace{-.2cm} \begin{array}{c}\aabb^{i_2}_{[i_1i_3]}\end{array}\hspace{-.2cm} 
	+ \hspace{-.1cm} \begin{array}{c}\aabb^{i_2}_{i_3}\end{array} \hspace{-.2cm}\ab_{{i_1}} \otimes  \hspace{-.2cm} \begin{array}{c}\aabb^{i_1}_{[i_2i_3]}\end{array}
\end{eqnarray*}
The algebra $\mathcal{H}^A_+$ together with this coproduct is a connected graded bialgebra. Next we consider the space of ladder trees, which we will denote by $\ell_{i_1 \cdots i_l}$. The vertices are decorated successively starting from the leaf, decorated by $i_1$, down to the root, which is decorated by $i_l$. 

\begin{proposition} \label{prop:laddertrees}
The space spanned by ladder trees forms a Hopf subalgebra in $\mathcal{H}^A_{+}$ with the coproduct
$$
	\delta^+(\ell_{i_1 \cdots i_l})=
	\sum_{I_1, \ldots, I_n \in A^* \atop I_1 \cdots I_n=i_1 \cdots i_l} \ell_{I_1} \cdots \ell_{I_n} \otimes \ell_{[I_1] \cdots [I_n]}.  
$$
The sum runs over all partitions of the decoration sequence $i_1 \cdots i_l \in A^*$ into blocks $I_j \in A^*$ such that the concatenation $I_1 \cdots I_n = i_1 \cdots i_l$.  
\end{proposition}

\begin{proof}
The partition into blocks of the decoration sequence $i_1 \cdots i_l$ corresponds to contractions on the ladder tree $\ell_{i_1 \cdots i_l}$. We will later see that these partitions can be understood in terms of compositions of integers. Recall that for single letters we have $[i_j]=i_j$.  
\end{proof}

The so-called substitution bialgebra $\mathcal{H}_+$ corresponding to undecorated rooted trees appeared in \cite{CEFM11}, where it was introduced in relation to backward error analysis on $B$-series. It is based on the seminal works \cite{CHV07,CHV10} and aims at providing an algebraic description of the effect of substituting a vector field corresponding to an infinitesimal character in the dual $\mathcal{H}^*_+$ into another $B$-series. It was shown in \cite{CEFM11} that there exists a left $\mathcal{H}_+$-bicomodule structure on the Butcher--Connes--Kreimer Hopf algebra $\mathcal{H}_{BCK}$, that is, $\Phi$: $\mathcal H_{BCK} \to \mathcal H_+ \otimes \mathcal H_{BCK}$, such that for $\tau \not= \mathbf{1}$ 
$$
	\Phi(\tau)=\delta^+(\tau) \in \mathcal{H}_+ \otimes \mathcal H_{BCK},
$$ 
and $\Phi(\mathbf{1})=\ab\otimes \mathbf{1}$. Let $\varphi$ be a character of $\mathcal H_+$, let $\alpha$ be any linear map from $\mathcal H_+$ into $k$, and let $b,c$ be linear maps form $\mathcal H_{BCK}$ into $k$. Let $\varepsilon$ be the co-unit of $\mathcal H_{BCK}$ and $d_{\ab}$ on $\mathcal H_{BCK}$ the infinitesimal character corresponding to the one-vertex rooted tree $\ab$, and $Z_{\ab}$ is the analogous infinitesimal character on $\mathcal H_+$. Then:
 \allowdisplaybreaks{
\begin{eqnarray*}
	\alpha \circledast  \varepsilon		&=& \alpha(\ab)\varepsilon\\
	\alpha \circledast  Z_{\ab}			&=& Z_{\ab} \circledast \alpha=\alpha\\
	Z_{\ab} \circledast  b			&=& b\\
	\varphi \circledast  (b \ast c)	&=&(\varphi \circledast  b)\ast(\varphi \circledast  c).
\end{eqnarray*}}
Here $\ast$ is the convolution product on the dual of $\mathcal H_{BCK}$ and $\circledast$ denotes the convolution product on $\mathcal{H}^*_+$, defined in terms of the coproduct \eqref{eq:subCoprod}, as well as the left action of $\mathcal H_+$ on $\mathcal H_{BCK}$. These results are generalised to the $A$-decorated case $\mathcal{H}^A_+$. The most important consequence of the interplay of the substitution bialgebra and the Butcher--Connes--Kreimer Hopf algebra is the following:

\begin{lemma}\cite{CEFM11}\label{sub:action}
Let $v$ be a linear map from $\mathcal{H}^A_{+}$ to $k$. Define the $\mathcal{H}^A_{BCK}$ endomorphism induced by $v$
$$
	\Psi_v := (v \otimes \mathrm{id}) \circ \Phi.
$$ 
If $v$ is a character of $\mathcal{H}^A_+$, then $\Psi_v$ is a Hopf algebra automorphism on $\mathcal{H}^A_{BCK}$.
\end{lemma}

Observe that the character $v$ on $\mathcal{H}^A_+$ defined to map $\ab_i$ to one, for all $i \in A$, and $v(\hspace{-.2cm} \begin{array}{c}\scalebox{0.5}{\aabb}^{i}_{j}\end{array}\hspace{-.2cm}):=\delta_{ij}\frac{1}{2}$, and any other tree to zero, gives for instance
\begin{eqnarray*}
	v(\!\!\begin{array}{c}\aabb^i_j\end{array}\!\!\!)
	&=& \begin{array}{c}\aabb^i_j\end{array} 
	+  \delta_{ij}\frac{1}{2} \begin{array}{c}\ab_{[ij]}\end{array} \\
	v(\hspace{-.3cm}\begin{array}{c}{\phantom{\aabb}}^{i_1}\aababb^{i_2}_{\!\!\!i_3}\end{array}\hspace{-.2cm})
	&=& \begin{array}{c}{\phantom{\aabb}}^{i_1}\aababb^{i_2}_{\!\!\!i_3}\end{array} 
	+  \delta_{i_1i_3}\frac{1}{2} \begin{array}{c}\aabb^{i_2}_{[i_1i_3]}\end{array}
	+  \delta_{i_2i_3}\frac{1}{2} \begin{array}{c}\aabb^{i_1}_{[i_2i_3]}\end{array}
\end{eqnarray*}


\section{Duality and substitution}
\label{sect:duality}
It is convenient to have a purely algebraic description of the flow, for this purpose we 
define the formal flow map $\varphi \in \mathcal{H}^{\Omega *}_{BCK} \bar{\otimes} \mathcal{H}^{\Omega}_{BCK}$
\[
	\varphi := \sum_{\tau \in \mathcal{T}_{\Omega}} \frac{1}{\sigma(\tau)} \tau \otimes \tau.
\]  
Note the abuse of notation by writing $\tau$ for $d_\tau$. Here $\mathcal{H}^{\Omega *}_{BCK} \bar{\otimes} \mathcal{H}^{\Omega}_{BCK}$ is the completed tensor product, i.e. the space of infinite series of tensor products of forests with Grossman-Larson product on the left and Butcher-Connes-Kreimer on the right, with the inverse limit topology comprising open sets generated by sequences agreeing up to a given order, see \cite{EFLMK,Reutenauer93}. 
The Taylor expansion \eqref{eq:CODEsolution} of a controlled differential equation with driving signal $X$ and vector fields $f_i$ takes the form
\[
	Y_t = \big(\mathfrak{F}_f\otimes X_{st}\big)(\varphi).
\]
Recall that for the empty tree $\mathbf{1} \in \mathcal{T}_{\Omega}$, we have $X^\mathbf{1}_{st}=1$ and $\mathfrak{F}_f[\mathbf{1}]=\mathrm{id}$. 

Given a pre-Lie algebra morphism $g$, we can attempt to find its adjoint $g^*$, defined such that
\[
	\big(\mathfrak{F}_f\otimes X_{st}\big)\big(g\otimes\mathrm{id}\big)(\varphi) 
	= \big(\mathfrak{F}_f\otimes X_{st}\big)\big(\mathrm{id}\otimes g^*\big)(\varphi).
\]
The original context of this problem was substitution of $B$-series. Indeed, suppose $X_{st}=(t-s)$ so that the rough differential equation is a standard ordinary differential equation. Consider increments of a fixed size, i.e., $X_{t,t+h}=h$; the iterated integrals then obey
\[
	\frac{1}{\sigma(\tau)} X_{t,t+h}^\tau = \frac{h^{|\tau|}}{\sigma(\tau)\tau !}.
\]
Letting $v(\tau):=\frac{1}{\tau!}$, the flow $(\mathfrak{F}_f\otimes X_{st})(\varphi)$ thus yields a formal series in powers of the step size $h$,
\[
	(\mathfrak{F}_f\otimes X_{st})(\varphi) 
	= \sum_{\tau \in \mathcal{T}} \frac{h^{|\tau|}}{\sigma(\tau)}\mathfrak{F}_f[\tau] v(\tau) = B(v,f),
\]
which is the Taylor series of the exact solution $Y_t$, i.e., the paradigm of a so-called $B$-series. See \cite{HLW06} for details. Replacing the term $v$ with a functional $a$ gives another $B$-series $B(a,f)$. Suppose then that the vector field $f$ is replaced by an expansion in terms of pre-Lie products of $f$, for instance
\[
	\tilde{f} = f + \frac{h^2}{2}f\triangleright f + \mathcal{O}(h^3).
\] 
Collecting the powers in $h$ appearing in $(\mathcal{F}_{\tilde{f}}\otimes X_{st})(\varphi)$ results in a new series, which may be related algebraically to the $B$-series $B(\frac{1}{\tau!},f)$. Indeed, $\tilde{f}=\frac{1}{h}B(f,a)$ for some infinitesimal character $a$ on $\mathcal{H}_{BCK}$, and we have the following

\begin{lemma}\cite{CEFM11}
Let $a$ be an infinitesimal character and $b$ a character, both defined over $\mathcal{H}_{BCK}$. The substitution of $B$-series is computed using the convolution product in $\mathcal{H}_+^*$:
\[
	B(b,\frac{1}{h}B(a,f))=B(a\circledast b,f).
\]
\end{lemma}

The above is generalised by first noting that it may be seen as the computation of an adjoint. Indeed, as $\mathfrak{F}_f$ is a pre-Lie algebra morphism, it follows that 
\[
	(\mathfrak{F}_{\tilde{f}} \otimes X_{t,t+h})(\varphi) 
	= (\mathfrak{F}_f\otimes X_{t,t+h})(g \otimes \mathrm{id})(\varphi)
\]
for the unique pre-Lie algebra morphism $g$ described by its action on the single node rooted tree
\[
	g(\ab) = \sum_{\tau'} \frac{a(\tau')  h^{|\tau'|-1}}{\sigma(\tau')}\tau'.
\]

The previous lemma then amounts to the computation of the adjoint of $g$. Bruned et al.~\cite{BCFP17} considered the problem of finding an adjoint in a more general rough path setting, establishing the following result.

\begin{lemma}\label{frizmachine}
Let $a$ be an infinitesimal character on $\mathcal{H}^A_{BCK}$ and define $g$ to be the pre-Lie algebra morphism
\[
	g(\ab_l) = \sum_{[\tau]=\ab_l} \frac{a(\tau)}{\sigma(\tau)} \tau,
\]
extended multiplicatively to forests. Then its adjoint $g^* = \Psi_a = a\circledast\mathrm{id} = (a\otimes\mathrm{id})\circ\Phi$ such that
\[
	\big(\mathfrak{F}_f\otimes X_{st}\big)\big(g\otimes\mathrm{id}\big)(\varphi) 
	= \big(\mathfrak{F}_f\otimes X_{st}\big)\big(\mathrm{id}\otimes \Psi_a\big)(\varphi).
\]
\end{lemma}


\section{Arborified Hoffman isomorphism}
\label{sect:HoffmanIso}

Hoffman \cite{Hoffman00} showed the existence of a unique and explicit Hopf algebra isomorphism $\log_H$ from $\mathcal{H}_{\star}$ to the shuffle Hopf algebra $\mathcal{H}_{\shuffle}$. Stratonovich iterated integrals of a Brownian path may be obtained from It\^{o} iterated integrals by an application of Hoffman's isomorphism \cite{EMPW15a,EMPW15b}. 

\smallskip

Recall the notion of composition of integers and the related concept of contracting words \cite{Hoffman00}. Let $A$ be a locally finite alphabet which is also a commutative semigroup with binary product $[-\ -]$ introduced earlier. A sequence $I:=(i_1,\ldots,i_m) $ of positive integers such that $i_1 + \cdots + i_m = n$, denotes a composition of the integer $n$. The set of those integer compositions is written $\mathcal{C}(n)$. The contraction of a word $w=a_{k_1} \cdots a_{k_n} \in A^*$ is denoted by $I[w]$ and defined in terms of the composition $I=(i_1,\ldots,i_m) \in \mathcal{C}(n)$ 
\[
	I[w]:=[a_{k_1} \cdots a_{k_{i_1}}][a_{k_{i_1+1}} \cdots a_{k_{i_2}}] \cdots 
	[a_{k_{i_m + 1}} \cdots a_{k_{i_n}}] \in A^*.
\]
Here $[i]:=i$. Hoffman's exponential \cite{Hoffman00,HoffmanIhara16} is defined as the map from $\mathcal H_\shuffle$ to $\mathcal H_\star$
\begin{align}
\label{HoffExp}
	\exp_H(u)&:=\sum_{I=(i_1,\ldots ,i_k)\in \mathcal{C}(|u|)} \frac{1}{i_1! \cdots i_k!}I[u],
\end{align}
for $u \in \mathcal H_\shuffle$. The sum ranges over all compositions of the integer $|u|$. Hoffman's logarithm is defined similarly from  $\mathcal H_\star$ to $\mathcal H_\shuffle$ 
\begin{align}
\label{HoffLog}
	\log_H(u)&:=\sum_{I=(i_1,\ldots ,i_k)\in \mathcal{C}(|u|)}\frac{(-1)^{n-k}}{i_1\cdots i_k}I[u].
\end{align}
For example, for $a_{k_1} \in A$ we have $\exp_H a_{k_1} =a_{k_1}$ and $\log_H a_{k_1} =a_{k_1}$. For words of length two and three we obtain
\begin{align*}
	\exp_H (a_{k_1}a_{k_2}) &=a_{k_1}a_{k_2}+\frac{1}{2}[a_{k_1} a_{k_2}],\\
	\log_H (a_{k_1}a_{k_2}) &=a_{k_1}a_{k_2}-\frac{1}{2}[a_{k_1} a_{k_2}],
\end{align*}
respectively
\begin{align*}
	\exp_H (a_{k_1}a_{k_2}a_{k_3}) &= a_{k_1}a_{k_2}a_{k_3}
	+\frac{1}{2!}[a_{k_1} a_{k_2}]a_{k_3}
	+\frac{1}{2!}a_{k_1} [a_{k_2}a_{k_3}]
	+\frac{1}{3!}[a_{k_1} a_{k_2} a_{k_3}],\\
	\log_H (a_{k_1}a_{k_2}a_{k_3}) &=a_{k_1}a_{k_2}a_{k_3}
	-\frac{1}{2}[a_{k_1} a_{k_2}]a_{k_3}
	-\frac{1}{2}a_{k_1} [a_{k_2} a_{k_3}]
         +\frac{1}{3}[a_{k_1} a_{k_2} a_{k_3}].\\
\end{align*}

The Hoffman exponential and logarithm are part of a class of quasi-shuffle algebra automorphisms induced by formal power series \cite{HoffmanIhara16}. See also \cite{NoPaTh13}. Letting $f=\sum_{n>0} f_n t^n$, we define
\[
	\psi_f(w) := \sum_{I=(i_1,\ldots,i_m)\in\mathcal{C}(|w|)} f_{i_1}\cdots f_{i_m} I[w].
\]
Observe that $f=t$ yields $\psi_f(w) = w$ whereas $f= t + \frac{1}{2}t^2$ gives
$$
	\psi_f(w) = \sum_{I=(i_1,\ldots,i_k)\in\mathcal{C}(|w|)} \frac{1}{2^{|I|_2}} I[w],
$$ 
where ${|I|_2}$ is the number of $i_j=2 \in I$.

\begin{lemma} \cite{HoffmanIhara16} Let $f,g$ be formal power series. The composition of their associated quasi-shuffle algebra automorphisms is then computed using the composition $f\circ g$ of formal power series:
\[
	\psi_f\circ\psi_g= \psi_{f\circ g}.
\]
\end{lemma}

\noindent If $f$ is invertible we may assume that $f_1=1$, such that $\psi_f \circ\psi_{f^{-1}}= \psi_{t}=\mathrm{id}$. Hoffman's exponential \eqref{HoffExp} and logarithm \eqref{HoffLog} follow from the regular exponential and logarithm. 

\begin{theorem}\cite{Hoffman00}
Hoffman's exponential, $\psi_{\exp(t)-1}$, is a Hopf algebra isomorphism $\mathcal{H}_{\shuffle} \rightarrow \mathcal{H}_*$, with inverse $\psi_{\log(1+t)}: \mathcal H_\star \to \mathcal H_\shuffle$.
\end{theorem}

Recall that Lemma~\ref{sub:action} states that the map $\Psi_v=(v\otimes\mathrm{id})\circ\Phi$ associated to a character $v$ of $\mathcal{H}^A_+$ is a Hopf algebra automorphism of $\mathcal{H}^A_{BCK}$. Moreover, suppose that $v(\tau)$ is independent of the decorations of $\tau$, and let $v(\ell_n)=f_n$, where $\ell_n$ is a ladder tree with $n$ nodes and the $f_n$ are coefficients of a formal power series $f$. Proposition~\ref{prop:laddertrees} then shows that when restricted to ladder trees, we have $\Psi_v = \psi_f$. The following result can then be considered an extension of Hoffman's theorem from quasi-shuffle automorphisms induced by formal power series to Butcher--Connes--Kreimer automorphisms induced by the substitution bialgebra:

\begin{lemma}\label{lem:Hoff1} Let $u,v$ be linear maps from $\mathcal{H}^A_{+}$ to $k$. Then 
$\Psi_u \circ \Psi_v = \Psi_{v \circledast u}$.
\end{lemma}

\begin{proof}
We calculate 
\begin{align*}
	\Psi_u \circ \Psi_v 
	&=(v \otimes u \otimes \mathrm{id})(\mathrm{id}\otimes \Phi)\Phi\\
	&=(v \otimes u \otimes \mathrm{id})(\delta^+  \otimes \mathrm{id})\Phi
	= \Psi_{v \circledast u} .		 
\end{align*}
\end{proof}

Recall that $\mathcal{H}^A_+$ is a graded but not connected bialgebra. Let $v$ be a multiplicative map from $\mathcal{H}^A_+$ to $k$ that maps $\ab_i$ to one, i.e., $v(\ab_i)=1$ for all $i \in A$. It turns out that in this case we can invert those characters by composition with the pseudo-antipode map $\alpha:=\sum_{i \ge 0} (\mathrm{id}-\eta\varepsilon)^{\circledast i}$, where $\eta$ and $\varepsilon$ are the unit respectively counit of $\mathcal{H}^A_+$. One can then show that $v \circledast (v \circ \alpha) = (v \circ \alpha) \circledast v = \varepsilon$, such that  $\Psi_{v \circ \alpha} \circ \Psi_v = \Psi_{\varepsilon}=\mathrm{id}$.

\medskip

Recall that a commutative and associative product on a decoration set $A$ gives rise not only to a quasi-shuffle algebra, but also to a contracting arborification, that is, a surjective algebra morphism $\mathfrak{a}^c: \mathcal{H}^A_{BCK} \rightarrow \mathcal{H}_\star$. The non-contracting arborification $\mathfrak{a}:\mathcal{H}_{CK} \rightarrow \mathcal{H}_{\shuffle}$ results from taking the product on $A$ to be trivial, so that the quasi-shuffle product becomes a shuffle product. 

Our central result is the description of the arborified Hoffman isomorphism, i.e., a canonical Butcher--Connes--Kreimer automorphism $\Psi_v$ making the following diagram commute:
\begin{equation}
\label{cdiag}
\begin{diagram}
	\mathcal{H}^A_{BCK}   	&& \rTo{\Psi_v} 		&& \mathcal{H}^A_{BCK} \\
	\dTo{\mathfrak{a}} 		&& 			 	&& \dTo{\mathfrak{a}^c} \\
	\mathcal{H}_{\shuffle}  	&& \rTo{\exp_H}  	&& \mathcal{H}_\star  
\end{diagram}
\end{equation}
\begin{theorem}\label{main:theorem}
Let $\tau \in \mathcal{T}_A$ be an $A$-decorated rooted tree. Define $v(\tau):=\frac{1}{\tau!}$ be the inverse tree factorial, which is multiplicative and $v(\ab_i)=1$ for all $i \in A$. Then $\Psi_v$ is the arborified Hoffman exponential, i.e., it is the unique Butcher--Connes--Kreimer Hopf algebra automorphism satisfying the following:
\begin{enumerate}
 	\item $\Psi_v$ makes the diagram \eqref{cdiag} commute.
 	\item The adjoint of $\Psi_v$ is a pre-Lie morphism.
\end{enumerate}
In fact, the adjoint of the arborified Hoffman exponential is the pre-Lie morphism defined on single vertex trees by
\begin{equation}
\label{eq:dualArboHoffm}
	\Psi^*_v(\ab_i) = \sum_{[\tau]=\ab_i} \frac{\mathrm{cm}(\tau)}{|\tau|!} \tau.
\end{equation}
\end{theorem}

Before we prove this theorem, we show how the arborified Hoffman exponential generalises the classical Hoffman exponential. Indeed, consider the particular case of the Hopf subalgebra of ladder tree $\ell_{i_1 \cdots i_l}$ in $\mathcal{H}^A_{BCK} $. Recall Proposition \ref{prop:laddertrees} and the fact that $\ell_{i_1 \cdots i_l}!=l!$. Then we find that 
\begin{align*}
	\Psi_v(\ell_{i_1 \cdots i_l}) 
	&= (v \otimes \mathrm{id})\Phi(\ell_{i_1 \cdots i_l})\\
	&= \sum_{I_1, \ldots, I_n \in A^* \atop I_1 \cdots I_n=i_1 \cdots i_l} 
	v(\ell_{I_1}) \cdots v(\ell_{I_n}) \ell_{[I_1] \cdots [I_n]}\\
	&= \sum_{I_1, \ldots, I_n \in A^* \atop I_1 \cdots I_n=i_1 \cdots i_l} \frac{1}{\ell_{I_1}! \cdots \ell_{I_n}!} \ell_{[I_1] \cdots [I_n]},
\end{align*}
in accordance with the remarks preceding Lemma~\ref{lem:Hoff1} and the fact that $I_1 \cdots I_n=i_1 \cdots i_l$ is equivalent to the application of the composition $I=(i_1,\ldots,i_n)$ to the word $w=i_1 \cdots i_l$, where $i_m$ is the number of letters, that is, the length of the interval $I_m$, for $ m=1, \ldots, n$.

Let us look at the tree ${}_{\phantom{i_1}}^{i_2}\hspace{-.08cm} \scalebox{0.5}{\aababb}^{\!i_3}_{\!\!i_1}$ as another example.   
$$
	\Psi_v( \hspace{-.3cm}\begin{array}{c}{\phantom{\aabb}}^{i_1}\aababb^{i_2}_{\!\!\!i_3}\end{array}\hspace{-.2cm})
	= \frac{1}{3}\ab_{[i_1 i_2i_3]} 
	+ \hspace{-.3cm}\begin{array}{c}{\phantom{\aabb}}^{i_1}\aababb^{i_2}_{\!\!\!i_3}\end{array}\hspace{-.2cm}
	+ \frac{1}{2} \hspace{-.2cm} \begin{array}{c}\aabb^{i_2}_{[i_1i_3]}\end{array}\hspace{-.2cm} 
	+ \frac{1}{2}  \hspace{-.2cm} \begin{array}{c}\aabb^{i_1}_{[i_2i_3]}\end{array}\hspace{-.1cm} .
$$
Now, we compare $\mathfrak{a}^c(\Psi_v({}_{\phantom{i_1}}^{i_2}\hspace{-.08cm} \scalebox{0.5}{\aababb}^{\!i_3}_{\!\!i_1}))$ with $\exp_H(\mathfrak{a}({}_{\phantom{i_1}}^{i_2}\hspace{-.08cm} \scalebox{0.5}{\aababb}^{\!i_3}_{\!\!i_1}))$, which should be equal if the diagram commutes. 
$$   
	\mathfrak{a}^c(\Psi_v(\hspace{-.3cm}\begin{array}{c}{\phantom{\aabb}}^{i_1}\aababb^{i_2}_{\!\!\!i_3}\end{array}\hspace{-.2cm}))=
	\frac{1}{3}{[i_1 i_2i_3]} 
	+ i_1i_2i_3 + i_2i_1i_3 + [i_1i_2]i_3 
	+ \frac{1}{2} i_2[i_1i_3] 
	+ \frac{1}{2} i_1[i_2i_3] . 
$$
Arborification gives $\exp_H(\mathfrak{a}({}_{\phantom{i_1}}^{i_2}\hspace{-.08cm} \scalebox{0.5}{\aababb}^{\!i_3}_{\!\!i_1}))=\exp_H( i_1i_2i_3 + i_2i_1i_3)$, which yields
\begin{align*}
	\exp_H(\mathfrak{a}(\hspace{-.3cm}\begin{array}{c}{\phantom{\aabb}}^{i_1}\aababb^{i_2}_{\!\!\!i_3}\end{array}\hspace{-.2cm}))
	&= i_1i_2i_3 
	+ \frac{1}{2} [i_1i_2]i_3 
	+ \frac{1}{2} i_1[i_2i_3] 
	+ \frac{1}{3!} [i_1i_2i_3] \\
	&
	+ i_2i_1i_3 
	+ \frac{1}{2} [i_2i_1]i_3 
	+ \frac{1}{2} i_2[i_1i_3] 
	+ \frac{1}{3!} [i_2i_1i_3] \\
	&= \frac{1}{3}{[i_1 i_2i_3]} 
	+ i_1i_2i_3 + i_2i_1i_3 + [i_1i_2]i_3 
	+ \frac{1}{2} i_1[i_2i_3] 
	+ \frac{1}{2} i_2[i_1i_3]. 
\end{align*}
In the last equality we used associativity and commutativity of the semigroup $A$. 

\begin{proof}
To prove Theorem \ref{main:theorem}, we first note that Lemma~\ref{sub:action} guarantees that $\Psi_v$ is a Butcher--Connes--Kreimer automorphism. 
As the Hoffman exponential is an algebra morphism, to prove commutativity of the diagram \eqref{cdiag} it suffices to show that
\begin{equation}
\label{eq:ArboHoff}
	\mathfrak{a}^c\circ \Psi_v = \exp_H \circ\, \mathfrak{a}.
\end{equation}
Recall the definitions of (contracting) arborification in \eqref{arborifc} respectively  \eqref{arborif}. Note that contractions on the right-hand side of \eqref{eq:ArboHoff} solely stem from Hoffman's exponential $\exp_H$, whereas on the left-hand side they come from both the contracting arborification as well as $\Psi_v$. Following Hoffman, we proceed by counting. In particular, both the left-hand side and the right-hand side are sums over the same words, with coefficients. It remains to match these coefficients. To see this, assume that $\tau \in \mathcal{T}_A$ has $n$ vertices decorated by ${i_1},\ldots, {i_n} \in A$, and note that from the definition of the map $\mathfrak{a}$ in \eqref{arborif} we find that 
\[
	\mathfrak{a}(\tau) = \sum_{\mathcal{N}(\tau)} {i_1} \cdots {i_n},
\]
where the sum is over all ways of ordering the decorations of the nodes of the tree into a word in such a way that it respects the partial order of the tree. The definition of $\mathfrak{a}$ implies that rightmost letter of ${i_1} \cdots {i_n}$ is the decoration of the root of $\tau$. Similarly, we have
\[
	\mathfrak{a}^c(\tau) = \sum_{I,\mathcal{M}(\tau)} I[{i_1}\cdots {i_b}],
\]
where the sum is over all words ordered as above, and over all compositions of these words such that only letters which are incomparable with respect to the partial order of $\tau$ may be contracted. Note that we count only once each contracted block $[i \cdots j]$, and not all the different ways this may arise from different orderings. Now the right-hand side of the identity may be written
\[
	\exp_H(\mathfrak{a}(\tau)) = \sum_{K,\mathcal{N}(\tau)} c_K K[{i_1} \cdots {i_n}]
\]
where $c_K$ are the Hoffman exponential coefficients, i.e., the usual factorials $\frac{1}{|w|!}$ for each composed word $[w]$ with $|w|$ letters. We now consider the left-hand side. Note that following arborification, each contraction from $\Psi_v$ acting on $\tau$ will correspond to a composition $J$ of the letters marking the contracted vertices. It follows that
\[
	\mathfrak{a}^c(\Psi_v(\tau)) = \sum_{I,J,\mathcal{M}(\tau)} c_{I,J} I\circ J[{i_1} \cdots {i_b}],
\]
for some coefficient $c_{I,J}$. Suppose that a subword $w= i \ldots j$ is composed in the above sum. The partial ordering on the vertices gives rise to a forest $\tau_1\cdots \tau_m$; this composition could only arise by having $J$ contract each of the trees $\tau_i$, and then letting $I$ compose the remaining (incomparable) vertices. It is clear then that the above sum may be rewritten
\[
	\mathfrak{a}^c(\Psi_v(\tau)) = \sum_{L,\mathcal{M}(\tau)} c_L L[{i_1}\cdots {i_b}],
\]
where the sum is over all compositions $L$, and the coefficient $c_L$ for each composed word $w$ with partial order $\tau_1\cdots\tau_m$ is $\frac{1}{\tau_1!\cdots\tau_m!}$. To relate the two, it remains to compare the sums over $\mathcal{M}$ and $\mathcal{N}$, i.e., to account for the fact that the right-hand side of the identity to be proven allows multiple ways of obtaining the same composed word, that is, a letter $[w]$. For this purpose, we note that given such a $[w]$, the number of times it appears in the sum from different orderings is equal to the number of linear extensions of the underlying partial order, i.e., number of ways of extending to a total order. Letting the forest of the order be $\tau_1\cdots\tau_m$, and the total number of letters of $w$ be $n$, this is exactly (see \cite{Knuth})
\[
	\frac{(n+1)!}{[\tau_1\cdots\tau_m]!}=\frac{(n+1)!}{(n+1)\tau_1!\cdots\tau_m!} 
	= \frac{n!}{\tau_1!\cdots\tau_m!},
\]
where the first equality comes from the recursive definition of the tree factorial. As multiplying $c_K$ (the Hoffman coefficients) by the right hand side above gives $c_L$, we are done.

The form of the adjoint follows from Lemma~\ref{frizmachine}, as 
\[
	\frac{1}{\sigma(\tau)\tau!} = \frac{\mathrm{cm}(\tau)}{|\tau|!}.
\]
The proof is then concluded by observing that any morphism of the free pre-Lie algebra $\mathcal{P}(A)$ is determined by the values it takes on the set of decorations $A$. 
\end{proof}


\section{Marcus canonical extension}
\label{sect:Marcus}

A wide variety of physical phenomena are well approximated by stochastic differential equations (SDEs) of the form
\begin{equation}
\label{SDE}
	X_t = X_0 + \int_0^t a(X_s)ds + \sum_{j=1}^N \int_0^t b_j(X_{s-})dZ^j_s,
\end{equation}
where $\{Z^1_t,\ldots,Z^N_t\}$ are $N$ real-valued semimartingales \cite{Protter05} and $a,b_1, \ldots, b_N$ are vector fields in $\mathbb{R}^d$. To ungarble notation, we suppose in the following that $N=1$, i.e., we consider the simple case of a single real-valued semimartingale. Often processes such as white noise, with its associated Wiener process differential $dW_t$, are approximations or limits of a more regular noise $dW^{\epsilon}_t$. Suppose that the equation describing the physical system is then an ordinary differential equation (ODE) of the form 
\begin{equation}
\label{regODE}
	dX^{\epsilon}_t = a(X^{\epsilon}_t)dt + b(X^{\epsilon}_t)dW^{\epsilon}_t.
\end{equation}   
In general, the solution $X^{\epsilon}_t$ does not converge to $X_t$ as the regular noise $dW^{\epsilon}_t$ approaches $dW_t$. In the case of SDEs driven by Wiener processes, this problem is usually resolved with the help of Stratonovich integration: the regular ODE \eqref{regODE} converges to the solution of the associated Stratonovich SDE
\[
	X_t = X_0 + \int_0^t a(X_s)ds + \int_0^t b(X_s)\!\circ\! dW_s.
\]
Indeed, the regularised equation does converge to an It\^{o} SDE, but with modified vector fields: $X^{\epsilon}_t \rightarrow X_t$, where
\begin{equation}
\label{contnoise}
	X_t = X_0 + \int_0^t \Big(a(X_s) 
			+\frac{1}{2}\sum_{i=1}^d b^i(X_s)\frac{\partial b(X_s)}{\partial x^i}\Big)ds 
			+ \int_0^t b(X_s)dW_s.
\end{equation}
McShane \cite{McShane75} called this equation the canonical extension. Observe that the term $\sum_{i=1}^db^i(X_s)\frac{\partial b}{\partial x^i}(X_s)$ can be interpreted as the pre-Lie product  $(b \triangleright b)(X_s)$ of the vector field $b$ with itself defined in \eqref{preLieVect1}. Indeed, it turns out that Marcus' main result \cite{Marcus78} is that for SDEs driven by general semimartingales, the regularised solutions $X^{\epsilon}_t$ tend to the solution of the following canonical extension:
\begin{equation}\label{Marcus}
	X_t = X_0 + \int_0^t a(X_s)ds 
	+ \sum_{n=1}^{\infty} \frac{1}{n!}\int_0^t 
	\big( b\triangleright(\cdots b \triangleright(b\triangleright b)) \big)(X_s)d[Z]^{(n)}_s,
\end{equation}
where $[Z]^{(1)}:=Z$, and $[Z]^{(n)}:=[Z,[Z]^{(n-1)}]$, $n>1$, is the $n$-fold quadratic variation  \cite{Protter05}. The vector fields $L^{n-1}_{b \triangleright}(b) = b\triangleright(\cdots b \triangleright (b\triangleright b))$ are the $n$-fold pre-Lie products of the vector field $b$ with itself. Observe that \eqref{contnoise} corresponds to $d[W]^{(1)}:=dW$, $d[W]^{(2)}=dt$, and $d[W]^{(n)}=0$ for $n>2$.

Marcus showed explicitly that in the particular case where $Z_t$ is a Poisson process $N_t$, as $[N]_t^{(n)} = N_t$ the canonical extension reduces to
\[
	X_t = X_0 + \int_0^t a(X_s)ds + \int_0^t \big( e^{L_{b\triangleright}}\circ \mathrm{id} 
	- \mathrm{id}\big)(X_s)dN_s.
\]
This has a rather intriguing analytic interpretation that has dominated the study of canonical extensions since their introduction. Indeed, the term $e^{L_{t b\triangleright}}$ is the flow corresponding to the differential equation $\dot{y}(t) = b(y(t))$. See equation \eqref{preLieVect2} for details. The second integral on the right is then a jump process that jumps along the integral curves of the vector field $b$ at the jump times of $N_t$. For processes with varying jump heights the distance moved along the integral curve is proportional to the jump size. A further layer of interpretation is that a fictitious time $t'$ is introduced that stretches at the jump times to accommodate the jumps; in this stretched time the process $X_t$ traverses the integral curve of $b$. See reference \cite{Applebaum09} for details. This interpretation has been used by Friz et al.~in \cite{FrizShekhar12} to construct a theory of rough paths for L\'evy processes. The analysis is considerably more complicated for general semimartingales; here the canonical extension cannot be readily collapsed to give integral curves, and this interpretation may not exist.

We will interpret \eqref{Marcus} in terms of substitution. Indeed, expand the decoration set $\{0,1\}$ corresponding to $N=1$ to the algebra $A=\mathbb{Z}_{\ge 0}$, where $0$ corresponds to the drift term, and $n>0$ to the $n$-fold quadratic variation $[Z]^{(n)}$. The vector fields $f_i$ are defined to be zero for $i>1$. The Marcus extension can be interpreted as a map $\Psi^*_v: \mathcal{F}_A \rightarrow \mathcal{F}_A$, where $v$ is the character multiplying by the inverse tree factorial, for which
\[
	X_t = (\mathfrak{F}_f\otimes Z_{0t})(\Psi^*_v\otimes\mathrm{id})(\varphi)
\]

\begin{theorem}[Hoffman and Marcus]
The Marcus extension map is the adjoint of the arborified Hoffman exponential for the decoration set $A$.
\end{theorem}

\begin{proof}
This is an immediate consequence of (\ref{Marcus}), Lemma~\ref{cm} and the form of the adjoint arborified Hoffman exponential given in Theorem~\ref{main:theorem}.
\end{proof}

\begin{rmk}{\rm{The Marcus extension gives us a canonical rough path renormalisation in the following sense. Suppose the branched rough path lift is generated by integrals $X^{\tau}$ obeying a quasi-shuffle law. The arborified Hoffman exponential tells us that the integrals $X^{\Psi_v(\tau)}$ are geometric in the sense that their arborification gives a geometric rough path, i.e., the usual rules of calculus are obeyed. We then search for a modification $\tilde{f}$ of the driving vector field $f$ such that
\[
	\sum_{\tau\in \mathcal{T}_A} \frac{1}{\sigma(\tau)} 
	\mathfrak{F}_{\tilde{f}} (\tau) X^{\tau}
	= \sum_{\tau\in \mathcal{T}_A} \frac{1}{\sigma(\tau)} 
	\mathfrak{F}_f(\tau) X^{\Psi_v(\tau)}
\]
This is equivalent to asking for the adjoint of the mapping $\Psi_v$ given in \eqref{eq:dualArboHoffm}, which as observed above is exactly the Marcus extension. 
}}
\end{rmk}


\section{The Hairer--Kelly map}
\label{sect:HKmap}

An alternative approach to the construction of a canonical rough path is given by Hairer and Kelly. In \cite{HaiKel15} they associate to any branched rough path $X^{\tau}$ and driving vector fields $f$ a geometric rough path $\bar{X}^w$ and a new set of vector fields $\bar{f}$ such that the solutions of the associated rough differential equations coincide. As for the Marcus extension, the set of decorations is expanded, but note that the Marcus extension and Hairer--Kelly map are in some sense inequivalent. Indeed, the expanded decoration sets do not coincide.

\begin{definition}[Hairer--Kelly map]\label{HKmap}
The map $\psi(\tau) : \mathcal{T} \rightarrow (T(\mathcal{T}), \shuffle)$ is defined as the unique Hopf algebra morphism from $\mathcal{H}_{BCK}$ to the shuffle Hopf algebra obeying
\[
	\psi(\tau) = (\psi\otimes \gamma) \circ \Delta,
\]
where $\Delta$ is coproduct in $\mathcal{H}_{BCK}$ and $\gamma:=\mathrm{id}-\epsilon$ is the augmentation projector.
\end{definition}

For instance
\begin{align*}
	\psi( \hspace{-.1cm} \begin{array}{c}\aaabbb\end{array}\hspace{-.1cm})
	&= \hspace{-.1cm} \begin{array}{c}\aaabbb\end{array}\hspace{-.1cm} 
	+ \hspace{-.1cm}  \begin{array}{c}\aabb\end{array}\hspace{-.2cm} \otimes\hspace{-.2cm} \begin{array}{c}\ab\end{array} \hspace{-.1cm} 
	+ \hspace{-.1cm} \begin{array}{c}\ab \hspace{-.0cm}\otimes\hspace{-.0cm} \ab\end{array} \hspace{-.2cm}\otimes\hspace{-.2cm} \begin{array}{c}\ab\end{array}\hspace{-.1cm} 
	+ \hspace{-.1cm} \begin{array}{c}\ab\end{array} \hspace{-.2cm} \otimes\hspace{-.2cm}  \begin{array}{c}\aabb\end{array}\\
	\psi( \hspace{-.1cm} \begin{array}{c}\aababb\end{array}\hspace{-.1cm})
	&=\hspace{-.1cm} \begin{array}{c}\aababb\end{array}\hspace{-.1cm}
	+ \hspace{-.1cm} \begin{array}{c}\ab \shuffle \ab\end{array} \hspace{-.2cm}\otimes\hspace{-.2cm} \begin{array}{c}\ab\end{array}\hspace{-.1cm} 
	+ \hspace{-.0cm} 2 \hspace{-.1cm} \begin{array}{c}\ab\end{array} \hspace{-.2cm} \otimes\hspace{-.2cm}  \begin{array}{c}\aabb\end{array}\hspace{-.1cm} .
\end{align*}
\begin{rmk}{\rm{
It is interesting to note that the arborification and contracting arborification maps may be constructed using a recursion of the above form, where $\gamma$ is replaced with the map that acts as the identity on all trees comprising a single node, and evaluates to zero on all other trees. Setting the range to be the shuffle Hopf algebra results in the arborification morphism, whilst if the range is quasi-shuffle we obtain the contracting arborification.}}
\end{rmk}

Central to the Hairer--Kelly approach is the construction of a geometric rough path $\bar{X}$ for which $\bar{X}(\psi(\tau))=X^{\tau}$ for all rooted trees $\tau \in \mathcal{T} $. Moreover, the construction is symmetric in the sense that $\bar{X}(\tau_1\otimes \tau_2)=\bar{X}(\tau_2\otimes \tau_1)$. There is a natural map $\pi:T(\mathcal{T})\rightarrow\mathcal{F}$ from tensor products of trees to forests that maps $\tau_1\otimes \cdots \otimes\tau_n\mapsto \tau_1\cdots\tau_n$. Much of the algebraic manipulation underlying the Hairer--Kelly results can therefore be expressed in terms of the symmetrised map $\tilde{\psi}=\pi\circ\psi:\mathcal{F}\rightarrow\mathcal{F}$. It turns out that this can be characterised using the right comodule structure on forests $\mathcal{F}$ associated to the substitution bialgebra. 

\begin{lemma}[symmetrised Hairer--Kelly map]
The symmetrised Hairer--Kelly map is the map $\tilde{\psi}:\mathcal{F} \rightarrow \mathcal{F}$ defined on trees $\tau$ by
\[
	\tilde{\psi}(\tau)=(\mathrm{id}\otimes \mathrm{cm}\cdot\sigma) \circ \tilde\Phi,
\]
where $\tilde\Phi$ is the coaction defined with respect to the free contraction, i.e., contractions of any subtree in $\tau$ to a node decorated by contracted tree $\tau$. The extension to forests is by $\tilde\psi(\tau_1\cdots\tau_n):=n!\tilde\psi(\tau_1)\cdots\tilde\psi(\tau_n)$.
\end{lemma}
For example, we have
\begin{eqnarray*}
	\tilde{\psi}( \hspace{-.1cm} \begin{array}{c}\aababb\end{array}\hspace{-.1cm})
	&=& 
	\hspace{-.1cm} \begin{array}{c}\aababb\end{array} (\mathrm{cm}\cdot\sigma)(1) +
	\begin{array}{c}\ab \,\ab\, \ab\end{array} (\mathrm{cm}\cdot\sigma)(\hspace{-.1cm} \begin{array}{c}\aababb\end{array}\hspace{-.1cm}) +
	2 \hspace{-.1cm} \begin{array}{c}\ab \;\aabb\end{array}
	(\mathrm{cm}\cdot\sigma)(\begin{array}{c}\aabb\end{array}) \\
	&=& \hspace{-.1cm} \begin{array}{c}\aababb\end{array}
	+ 2 \begin{array}{c}\ab \,\ab\, \ab\end{array}
	+ 2 \hspace{-.1cm} \begin{array}{c}\ab \;\aabb\end{array} 
\end{eqnarray*}

\begin{proof}
To prove this, we note that the effect of the recursion in Definition~\ref{HKmap} is to sum over all (not necessarily admissible) cuts of the tree $\tau$ into (ordered) subforests $\tau_1 \otimes \cdots \otimes \tau_n$, where the $\tau_i$ are ordered such that their roots must respect the partial order of the vertices in $\tau$. Upon symmetrization, this becomes the sum over all (unordered) subforests $\tau_1 \cdots \tau_n$ with a combinatorial factor coming from the number of linear extensions of the partial order of the roots of the $\tau_i$ to a total order. This cut coincides with the identification of possible subforests to conduct a substitution; the tree that results on the right-hand side of $\delta^+$ from contracting $\tau$ according to the subforest $\tau_1 \cdots \tau_n$ is precisely the tree $\tau$ associated to the partial order of the roots of $\tau_i$. The number of linear extensions of $\tau$ is precisely $\mathrm{cm}(\tau)\sigma(\tau)=\frac{\tau!}{|\tau|!}$, see \cite{Knuth}, hence the result follows.
\end{proof}

The mapping $\tilde\psi : \mathcal{F} \rightarrow \mathcal{F}$ is invertible, and we will make use of the (compositional) inverse of $\tilde\psi$. The following recursion applies.

\begin{lemma}
Let $\widehat\delta^+(\tau) = \tau^{(1)} \otimes \tau^{(2)}$ be the reduced substitution coproduct, expressed in sumless Sweedler notation. The inverse map $\tilde\psi^{-1}$ is defined on trees by
\[
	\tilde\psi^{-1}(\tau) = \tau - \tilde\psi^{-1}(\tau^{(1)})\cdot (\mathrm{cm}\cdot\sigma)(\tau^{(2)}) 
\]
and extended to forests by $\tilde\psi^{-1}(\tau_1\ldots\tau_n) = \frac{1}{n!}\tilde\psi^{-1}(\tau_1)\cdots\tilde\psi^{-1}(\tau_n)$.
\end{lemma}

\begin{proof}
The above recursion is of a form similar to that defining the antipode in a connected graded Hopf algebra. Indeed it is a sort of `twisted' antipode. One can prove that this is the inverse of $\tilde\psi$, as we have
\[
	\tilde\psi(\tau) = \tau + \tau^{(1)}\cdot (\mathrm{cm}\cdot\sigma)(\tau^{(2)}), 
\]
and the term on the right is of strictly lower degree than the tree $\tau$. Any map of the form $\tilde\phi(\tau) = \tau + \tilde\phi_{-}(\tau)$ where $\tilde\phi_{-}(\tau)$ is of lower grade than $\tau$ has an inverse of the form $\tilde\phi^{-1}(\tau) = \tau -\tilde\phi^{-1}(\tilde\phi_-(\tau))$ . Applying this to the above gives the result.
\end{proof}

The Hairer--Kelly approach is based around the following equivalence,
\[
	\sum_{\tau \in \mathcal{F}} \frac{1}{\sigma(\tau)} \mathfrak{F}_f(\tau) X^{\tau}_{st}
	= \sum_{\tau\in \mathcal{F}} \frac{1}{\sigma(\tau)} \mathfrak{F}_f(\tilde\psi^*(\tau)) X^{\tilde\psi^{-1}(\tau)}_{st},
\]
which is an immediate consequence of the definition of the adjoint. Indeed, the intention is to define a new rough path $\bar{X}_{st}$ obeying $\bar{X}^{\tau}_{st} = X^{\tilde\psi^{-1}(\tau)}_{st}$. The key result is the following description of the flow map in terms of $\bar{X}$, which is the basis for considering $\bar{X}$ as a geometric rough path.

\begin{theorem}\cite{HaiKel15} 
The flow map $\phi_{st}$ for the rough differential equation driven by the vector fields $f_i$ and rough path $X$ may be written as a word series using the rough path $\bar{X}$ as follows
\[
	\phi_{st} = \sum_{\tau_1 \cdots \tau_n \in \mathcal{F}} \frac{1}{n!} 
	\mathfrak{F}_f ({\tau_1* \cdots * \tau_n}) \bar{X}_{st}^{\tau_1\cdots\tau_n}
\]
\end{theorem}

The rough path $\bar{X}$ is then interpreted as a geometric rough path as indicated in the above formula, by expanding the set of decorations to include trees. The usual `sewing' procedure allows us to consider only trees up to a given order depending on the regularity of the rough path $X$. As a geometric rough path $\bar{X}$ is symmetric, as follows from the description $\bar{X}^{\tau_1\cdots\tau_n} = X^{\phi^{-1}(\tau_1\cdots\tau_n)}$ and the symmetry of $\tilde\psi^{-1}$. The grading of $\bar{X}$ must then be based on word length, and sewing must be performed `horizontally'. Hairer and Kelly show that this is possible using the Lyons--Victoir extension theorem.

\begin{rmk}{\rm{
Note that the words do not correspond to iterated integrals of $X$, but rather polynomials, as can be seen from the definition via $\psi$. This is the reason for the symmetry of $\bar{X}$.}}
\end{rmk}

\begin{rmk}{\rm{
We conclude by noting that the Hairer--Kelly extension is not the arborification of the Marcus extension. Indeed, this is related to the observation of Hairer and Kelly that their map contains more terms than required in the case of Wiener processes. The Marcus canonical extension answers the question about how to construct simplifications in more general cases where a quasi-shuffle relation is still present.}} 
\end{rmk}


\end{document}